\documentclass[11pt]{article}
\usepackage{latexsym}
\usepackage{amsfonts}
\usepackage{amsmath}

\newtheorem{theorem}{Theorem}[section]

\newenvironment{proof}{\noindent {\bf Proof:\ }}{{\quad $\Box$ \medbreak}}
\title{Odd Holes in Bull-Free Graphs}
\author{Maria Chudnovsky
\thanks{Partially supported by NSF grant DMS-1550991 and by US Army Research 
Office grant W911NF-16-1-0404.}\\
Princeton University, Princeton, NJ 08544, USA
\and
Vaidy Sivaraman\\
Binghamton University, Binghamton, NY 13902, USA
}

\begin{document}
\maketitle
\begin{abstract}
The complexity of testing whether a graph contains an induced odd cycle of length at least five is currently unknown. In this paper we show that this can be done in polynomial time if the input graph has no induced subgraph isomorphic to the bull (a triangle with two disjoint pendant edges). 
\end{abstract}

\section{Introduction}

All graphs in this paper are finite and simple. 
The complement $G^c$ of $G$ is the graph with vertex set $V(G)$ and such that two vertices are adjacent in $G^c$ if and only if they are non-adjacent in $G$. 
For two graphs $H$ and $G$, $H$ is an {\em induced subgraph} of $G$ if 
$V(H) \subseteq V(G)$, and a pair of vertices $u,v \in V(H)$ is adjacent if and only if it is adjacent in $G$. We say that $G$ {\em contains} $H$ if $G$ has an induced subgraph isomorphic to $H$. If $G$ does not contain $H$, we say that $G$
is {\em $H$-free}. For a set $X \subseteq V(G)$ we denote by 
$G[X]$ the induced subgraph of $G$ with vertex set $X$. 
A path $P$ in a graph is a sequence $p_1-\ldots-p_k$ (with $k \geq 1$) of distinct vertices such that $p_i$ is adjacent to $p_j$ if and only if $|i-j|=1$. We say that the {\em length} of this path is $k-1$. We call $p_1$ and $p_k$ the {\em ends} of $P$, and write $P^*=V(P) \setminus \{p_1,p_k\}$.
A {\em hole} in a graph is an induced subgraph that is isomorphic to  the cycle $C_k$ with $k\geq 4$, and $k$ is the {\em length} of the hole. A hole is {\em odd} if $k$ is odd, and {\em even} otherwise. The vertices of a hole can be 
numbered $c_1, \ldots, c_k$ such that $c_i$ is adjacent to $c_j$ if and only if
$|i-j| \in \{1,k-1\}$; sometimes we write $C=c_1-\ldots-c_k-c_1$. 
An {\em antihole} in a  graph is an induced subgraph that is isomorphic to  
$C_k^c$ with $k\geq 4$, and again $k$ is the {\em length} of the antihole. 
Similarly, an  antihole is {\em odd} if $k$ is odd, and {\em even} otherwise.
The {\em bull} is the graph consisting of a triangle with two disjoint pendant edges. A graph is {\em bull-free} if no induced subgraph of it is isomorphic to the bull. 

A graph $G$ is called {\em perfect} if for every 
induced subgraph $H$ of $G$, $\chi(H) = \omega (H)$; and {\em Berge} if
it has no odd holes and no odd antiholes. In \cite{Bergealg} it 
was shown that:

\begin{theorem}
\label{Berge}
There is  an algorithm  that tests if an input graph $G$ is Berge in time
$O(|V(G)|^9)$.
\end{theorem}

However, the complexity of testing for an odd hole is still unknown. The main result of this paper is a solution of a special case of this problem, as follows.

\begin{theorem}
\label{main}
There is  an algorithm  that tests if an input bull-free graph $G$ contains an odd hole in time $O(|V(G)|^5)$.
\end{theorem}

Since testing if a graph $G$ is bull-free can be done in time  $O(|V(G)|^5)$
by enumerating all $5$-tuples of vertices, \ref{main} immediately implies:

\begin{theorem}
\label{robust}
There is an algorithm  that tests if an input  graph $G$ contains a bull or an odd hole in time $O(|V(G)|^5)$.
\end{theorem}

Given a graph $G$ and two disjoint sets $A,B \subseteq V(G)$ we say that
$A$ is {\em complete} to $B$ if every vertex in $A$ is adjacent to every vertex in $B$, and that $A$ is {\em anticomplete} to $B$ if every vertex in $A$ is non-adjacent to every vertex in $B$. If $|A|=1$, say $A=\{a\}$, we say that $a$ (instead of $\{a\}$) is complete (or anticomplete) to $B$. An edge is $A$-complete
(or $a$-complete) if both of its ends are complete to $A$.
A set $X \subseteq V(G)$ is a {\em homogeneous set} if $1<|X|<|V(G)|$ and every vertex 
of $V(G) \setminus X$ is either complete or anticomplete to $X$.
If $G$ contains a homogeneous set, we say that $G$ {\em admits a homogeneous set decomposition}.
A six-vertex graph is an {\em anchor} if it consists of a $4$-vertex induced 
path $P$, a vertex $c$ complete to $V(P)$, and a vertex $a$ anticomplete to $V(P)$. (The only adjacency that has not been specified is between $a$ and $c$, and 
so there are exactly two anchors.)

Here is the outline of the algorithm.
First we test, by enumerating all $5$-tuples, if the input graph contains $C_5$,
and so from now on we may assume that the input is bull-free and $C_5$-free.
The following is a structural result about bull-free graphs that follows  
easily from \cite{bulls1}

\begin{theorem}
\label{hset}
If a $C_5$-free bull-free graph contains an anchor, then it contains   
a homogeneous set.
\end{theorem}

There are standard techniques that allow us to reduce the problem of testing for an odd hole to graphs with no homogeneous sets. Consequently, in view of
\ref{hset}, it is enough to design an algorithm that detects an odd hole in a 
bull-free graph that does not contain an anchor.

A hole $C$ in a graph $G$ is {\em clean} if for every 
$v \in V(G) \setminus V(C)$, the set of neighbors of $v$ in $V(C)$ is contained in a two-edge path of $C$. A {\em shortest odd hole} in $G$ is  an odd hole of minimum length. We say that $G$ is {\em pure} if either it contains no odd 
hole, or it contains a shortest odd hole that is also clean. 
Is is not difficult to prove that:

\begin{theorem} 
\label{pure}
Every bull-free graph without an anchor is pure.
\end{theorem}

A ``jewel'' and a ``pyramid'' are two types of graphs that we will define later;
we will also show that:

\begin{theorem}
\label{jewelpyramid}
Every jewel and every pyramid  contains a bull, a $C_5$,  or an anchor. 
\end{theorem}

The following is Theorem 4.2 of \cite{Bergealg}:

\begin{theorem}
\label{shortestodd}
There is an algorithm with the following specifications.
\begin{itemize}
\item {\bf Input:} A graph $G$ with no induced subgraph that is a jewel or a 
pyramid.
\item {\bf Output:} A determination if $G$ has a clean shortest odd hole.
\item {\bf Running time:} $O(|V(G)|^4)$.
\end{itemize}
\end{theorem}

\ref{shortestodd} immediately implies 
\begin{theorem}
\label{pureodd}
There is an algorithm with the following specifications.
\begin{itemize}
\item {\bf Input:} A pure graph $G$ with no induced subgraph that is a jewel 
or a pyramid.
\item {\bf Output:} A determination if $G$ has an odd hole.
\item {\bf Running time:} $O(|V(G)|^4)$.
\end{itemize}
\end{theorem}

Combining \ref{pure}, \ref{jewelpyramid} and \ref{pureodd} we deduce:

\begin{theorem}
\label{noanchoralg}
There is an algorithm with the following specifications.
\begin{itemize}
\item {\bf Input:} A bull-free $C_5$-free graph $G$ that does not contain an anchor.
\item {\bf Output:} A determination if $G$ has an odd hole.
\item {\bf Running time:} $O(|V(G)|^4)$.
\end{itemize}
\end{theorem}

This paper is organized as follows.   In Section~2 we define jewels and pyramids, and prove \ref{pure} and \ref{jewelpyramid}. In Section~3 we introduce the necessary terminology from~\cite{bulls1} and prove~\ref{hset}. In Section~4 we explain why  \ref{noanchoralg} implies \ref{main}.

\section{Jewels, pyramids and shortest odd holes}

First we prove (a slight strengthening of)  \ref{pure}.

\begin{theorem} 
\label{pure1}
Every $C_5$-free bull-free graph without an anchor is pure. In fact, every shortest odd hole in such a graph is clean.
\end{theorem}

\begin{proof}
Let $G$ be a $C_5$-free bull-free graph that does not contain an anchor. 
We may assume that $G$ contains an odd hole, for otherwise $G$ is pure; let
$C$ be a shortest odd hole in $G$. Let $C=c_1-\ldots-c_k-c_1$. Then $k\geq 7$. 
We prove that $C$ is clean. Let $v \in V(G) \setminus V(C)$, and suppose that
$N(v)$ is not contained in a two-edge path of $C$. A {\em $v$-gap} is a path 
$R$ of $C$, such that $v$ is adjacent to the ends of $R$, $v$ has no neighbor in
$R^*$, and $|V(R)|>2$. A {\em $v$-stretch} is a maximal path of length at least two of $C$ all of whose vertices are complete to $v$.
Since $N(v)$ is not contained in a two-edge path of $C$,
every $v$-gap has length less than $k-2$, and so it follows from the fact that
$C$ is a shortest odd hole in $G$ that every $v$-gap has even length.  
Since every edge of $C$ is either $v$-complete or belongs to a $v$-gap, it follows that  there is an odd number of $v$-complete edges in $C$, and consequently there exists an odd $v$-stretch. If some $v$-stretch has length one, then
$G$ contains a bull, so there is a $v$-stretch of length at least three. Thus we may assume that $v$ is complete to $\{c_1,c_2,c_3,c_4\}$. But now 
$G[\{c_1, c_2, c_3, c_4,v,c_6\}]$ is an anchor, a contradiction. This proves~\ref{pure}.
\end{proof} 

Next we prove \ref{jewelpyramid}. We start with the necessary 
definitions.
A {\em pyramid} is a graph formed by the union of a triangle 
$\{b_1,b_2,b_3\}$ (called the {\em base} of the pyramid),
a fourth vertex $a$ (called its {\em apex}), and three paths $P_1,P_2,P_3$, satisfying:
\begin{itemize}
\item for $i = 1,2,3$, $P_i$ has ends $a$ and $b_i$
\item for $1 \le i < j \le 3$, $a$ is the only vertex in both $P_i,P_j$, and $b_ib_j$ is the only edge of $G$ between 
$V(P_i)\setminus\{a\}$ and $V(P_j)\setminus\{a\}$
\item $a$ is adjacent to at most one of $b_1,b_2,b_3$.
\end{itemize}

A {\em jewel} is a graph $H$ with vertex set $\{v_1,v_2,v_3,v_4,v_5\} \cup F$, 
such that $H[F]$ is connected, 
$v_1v_2, v_2v_3, v_3v_4,  v_4v_5, v_5v_1$ are edges, $v_1v_3,v_2v_4,v_1v_4$ 
are non-edges, and $v_2,v_3,v_5$ are $F$-anticomplete, and $v_1,v_4$ are not.
In \cite{Bergealg} jewels are referred to as ``configuration $\mathcal{T}_2$''.

We now prove \ref{jewelpyramid}, which we restate.

\begin{theorem}
\label{jewelpyramid1}
Every jewel and every pyramid  contains a $C_5$, a bull,  or an anchor. 
\end{theorem}

\begin{proof}
Let $H$ be a pyramid. With the notation of the definition of the pyramid, we may assume that $a$ is non-adjacent to $b_1,b_2$. For $i \in \{1,2\}$ let $c_i$ be the neighbor of $b_i$ is $P_i$, Then $c_1,c_2 \neq a$, and so $c_1$ is non-adjacent to $c_2$. But now $H[\{b_1,b_2,b_3,c_1,c_2\}]$ is a bull. 

Next let $H$ be a jewel. Since $F$ is connected and $v_1,v_4$ have neighbors in
$F$, it follows that there is a path $P$ from $v_1$ to $v_4$ with 
$P^* \subseteq F$. If $|V(P)|=3$ then $v_1-v_2-v_3-v_4-P-v_1$ is a $C_5$, and
if $|V(P)|=4$, then $v_1-v_5-v_4-P-v_1$ is a $C_5$, so we may assume that
$|V(P)| \geq 5$. Let $p$ be the neighbor of $v_1$ in $P$, and let $q$ be the 
neighbor of $v_4$ in $P$. Then there is $s \in P^* \setminus \{p,q\}$, and
$s$ is anticomplete to $\{v_1, v_2, v_3, v_4, v_5\}$. If $v_5$ is complete to 
$\{v_2,v_3\}$, then $H[\{v_1, v_2, v_3, v_4, v_5, s\}]$ is an anchor, and if $v_5$ is
anticomplete to $\{v_2,v_3\}$, then $H[\{v_1, v_2, v_3, v_4, v_5\}]$ is a $C_5$, so we may assume that $v_5$ is adjacent to $v_2$ and not to $v_3$. But now $H[\{v_1,v_2,v_5,p,v_3\}]$ is a bull. This proves \ref{jewelpyramid}. 
\end{proof}

\section{Anchors}
In this section we prove \ref{hset}. Here we rely on results of \cite{bulls1} that are stated in terms of trigraphs, rather than graphs. A trigraph is a concept generalizing graphs. While in a graph  a pair of vertices can be adjacent or non-adjacent, in a trigraph there are three possible kinds of pair: adjacent, non-adjacent and semi-adjacent, and a trigraph is a graph if it contains no semi-adjacent pairs. Since every graph is a trigraph, results from \cite{bulls1} apply in our setting. For a more formal discussion of trigraphs we refer the reader to \cite{bulls1}.

A graph (or trigraph) is called {\em elementary} if it does not contain an anchor. We need the following (Theorem 3.3 of \cite{bulls1}):

\begin{theorem}
\label{trianchor}
Let $G$ be a bull-free trigraph that is not elementary. Then either
\begin{itemize}
\item one of $G,G^c$ belongs to $\mathcal{T}_0$, or
\item one of $G,G^c$ contains a homogeneous pair of type zero, or
\item $G$ admits a homogeneous set decomposition.
\end{itemize}
\end{theorem}

The class $\mathcal{T}_0$ and homogeneous pairs of type zero are defined in 
Section~3 of \cite{bulls1}. To deduce \ref{hset} from \ref{trianchor} we make 
the following two observations.
First we observe that every member of the class $\mathcal{T}_0$ contains a semi-adjacent pair, and so the first outcome does not happen in $G$ since $G$ is a graph, rather than a trigraph. Second we note that every graph that admits a homogeneous pair of type zero contains a $C_5$. Now \ref{hset} follows.

\section{Homogeneous Sets}

In this section we show how to use homogeneous sets for the purposes of our 
algorithm.  If $X$ is a homogeneous set in a graph $G$, we define two new
graphs $G_1(X)$ and $G_2(X)$, as follows. $G_1(X)$ is obtained from $G$ by deleting all but exactly one vertex of $X$ (note that it does not make a difference which vertex of $X$ is not deleted), and $G_2(X)=G[X]$.

First we prove the following.

\begin{theorem}
\label{holeprime}
Let $G$ be a graph, and let $X$ be a homogeneous set in $G$. If $G$
contains an odd hole, then at least one of  $G_1(X)$ and $G_2(X)$ contains an 
odd hole.
\end{theorem}

\begin{proof}
Let $C$ be an odd hole in $G$. We may assume that $V(C) \not \subseteq X$.
If $|V(C) \cap X| \leq 1$, then $G_1(X)$ contains an odd hole. Thus
$|V(C) \cap X| \geq 2$, and $V(C) \setminus X \neq \emptyset$. It follows that
$V(C) \cap X$ is a homogeneous set in $C$, a contradiction. This proves \ref{holeprime}.
\end{proof}

We can now prove \ref{main}, which we restate.

\begin{theorem}
\label{main1}
There is  an algorithm  that tests if an input bull-free graph $G$ contains an odd hole in time $O(|V(G)|^5)$.
\end{theorem}

\begin{proof}
Here is the algorithm.
\begin{enumerate}
\item Test if $G$ contains $C_5$ by enumerating all $5$-tuples. If yes, stop and output: {\em ``$G$ contains an odd hole''}.
\item Test if $G$ contains a homogeneous set; and find one if it exists. 
\begin{enumerate}
\item If no homogeneous set exists, run the algorithm of \ref{noanchoralg} on $G$, 
output its output, and stop.
\item Else, let $X$ be the homogenous set that we found; recurse on
$G_1(X)$ and $G_2(X)$.
\end{enumerate}
\end{enumerate}

\noindent{\em Proof of correctness:}
After step~$1$ we may assume that $G$ is $C_5$-free. If $G$ does not admit a homogeneous set decomposition, then $G$ has no anchor, and so step~$2(a)$ 
works correctly. Thus we may assume that  $X$ is a homogeneous set in $G$.
Now both $G_1(X)$ and $G_2(X)$ are induced subgraphs of $G$, and therefore they are both bull-free and $C_5$-free.
By~\ref{holeprime} it is enough to test if $G_1(X)$ and
$G_2(X)$ contain an odd hole, which is done in step~$2(b)$. 
This completes the proof of correctness.
\\
\\
{\em Complexity analysis:}
Clearly step~$1$ takes time $O(|V(G)|^5)$. By~\cite{hsetalg} we can find a 
homogeneous set in time $O(|V(G)|)$. By \ref{noanchoralg} 
steps~$2(a)$ takes time $O(|V(G)|^4)$. Since 
$|V(G_1(X))|+|V(G_2(X))|=|V(G)|+1$, it follows that the recursion of 
step~$2(b)$  takes time $O(|V(G)|^5)$. Consequently the algorithm runs in time
$O(|V(G)|^5)$, as claimed.
\end{proof}

\section{Acknowledgment}
This research was conducted during a Graph Theory Workshop  at the McGill 
University Bellairs Research Institute, and we express our gratitude to the institute and to the organizers of the workshop. We are also grateful to Ch\'{i}nh T. Ho\`{a}ng for suggesting this problem to us.

\end{document}